\documentclass[a4paper]{article}

\usepackage[english]{babel}
\usepackage[utf8x]{inputenc}
\usepackage[T1]{fontenc}

\usepackage[a4paper,top=3cm,bottom=2cm,left=3cm,right=3cm,marginparwidth=1.75cm]{geometry}

\usepackage{mathptmx,amsmath,colonequals}
\usepackage{amssymb,gensymb,mathrsfs}
\usepackage{amscd,amsthm}
\usepackage{mathtools}
\usepackage{graphicx}
\usepackage[colorinlistoftodos]{todonotes}
\usepackage[colorlinks=true, allcolors=blue]{hyperref}
\usepackage[shortlabels]{enumitem}
\usepackage{extarrows}
\usepackage{tikz,tikz-cd}

\newtheorem{thm}{Theorem}[section]
\newtheorem{lem}[thm]{Lemma}
\newtheorem{cor}[thm]{Corollary}
\newtheorem{prop}[thm]{Proposition}
\newtheorem{conj}[thm]{Conjecture}
\newtheorem{ques}[thm]{Question}
\newtheorem{hyp}[thm]{Hypothesis}

\newtheorem{defn}[thm]{Definition}
\newtheorem{ex}[thm]{Example}

\newtheorem{rmk}[thm]{Remark}

\newcommand{\isom}{\cong}
\newcommand{\ins}{\subset}

\newcommand{\Ker}{\operatorname{Ker}}

\newcommand{\sprt}{\operatorname{supp}}
\newcommand{\Frac}{\operatorname{Frac}}

\newcommand{\inj}{\xhookrightarrow{}}

\newcommand{\residue}{\kappa}

\newcommand{\nonneg}{\mathbb{R}_{\geq 0}}
\newcommand{\Gammanonneg}{\Gamma_{\geq 0}}
\newcommand{\contposreal}{C(\mathbb{R}_{\geq 0}, \nonneg)}

\newcommand{\contu}{C(U, \nonneg)}
\newcommand{\leg}{\mathscr{L}}
\newcommand{\newt}{\mathscr{N}}
\newcommand{\LN}{\leg\!\newt}

\newcommand{\res}{\operatorname{res}}
\newcommand{\restr}[2]{\res_{#1}^{#2 \leq} }
\newcommand{\restrint}[1]{\res_{#1} }
\newcommand{\ebox}{\epsilon^{\text{box}}}
\newcommand{\ebaar}{\epsilon^{\text{bar}}}
\newcommand{\dbox}{\delta^{\text{box}}}
\newcommand{\dbaar}{\delta^{\text{bar}}}

\newcommand{\tateone}[1]{\mathcal{O}_{#1}\langle x^{1 / p^{\infty}}\rangle}
\newcommand{\malcevt}{[[t^{\Gammanonneg}]]}
\newcommand{\malcevtrat}{((t^{\Gamma}))}
\newcommand{\malcevp}{[[p^{\Gammanonneg}]]}
\newcommand{\malcevprat}{((p^{\Gamma}))}
\newcommand{\coperfx}{[x^{1 / p^{\infty}}]}
\newcommand{\coperfxpow}{[[x^{1 / p^{\infty}}]]}
\newcommand{\puiseuxxrat}{((x^{1 / \mathbb{N}}))}
\newcommand{\puiseuxx}{[[x^{1 / \mathbb{N}}]]}

\newcommand{\argnorm}{\operatorname{argnorm}}
\newcommand{\supp}{\operatorname{supp}}

\title{The perfectoid Tate algebra
has uncountable Krull dimension}
\author{Jack J Garzella}
\begin{document}

\maketitle

\begin{abstract}
Let \(K\) be a perfectoid field with 
pseudo-uniformizer \(\pi\).
We adapt an argument of Du in 
\cite{DuUncountable} to show that the perfectoid
Tate algebra \(K\langle x^{1 / p^{\infty}} \rangle\) has an uncountable chain
of distinct prime ideals.
First, we conceptualize Du's argument,
defining the notion of a 
\textit{Newton polygon formalism}
on a ring.
We prove 
a version of Du's theorem in 
the prescence of a sufficiently nondiscrete
Newton polygon formalism.
Then, we apply our framework to the perfectoid
Tate algebra via a "nonstandard" Newton polygon
formalism (roughly, the roles of the series
variable \(x\) and the pseudo-uniformizer \(\pi\) 
are switched).
We conclude a similar statement for multivatiate
perfectoid Tate algebras using the one-variable
case.
\end{abstract}

\section{Introduction}

Perfectoid rings have been extremely useful tools since 
their introduction by Scholze in \cite{ScholzePerfSpaces},
leading to many advances in 
the mixed characteristic theory of
commutative algebra 
(see \cite{AndreDirectSummand}, \cite{MaSchwedeBCM}, \cite{BhattIyengarMa}),
\(p\)-adic Hodge theory (\cite{ScholzeSurvey}, \cite{ScholzeHodgeRigid}, \cite{KedlayaLiu1})
and the Langlands program (e.g. \cite{CarianiScholze1}, \cite{ScholzeTorsion}, \cite{FarguesScholze}).
However, their study as algebraic objects is 
limited due to the fact that any perfectoid
ring which is not a field is non-noetherian \cite[\S 2, Corollary 2.9.3]{AWSNotes},
so the vast majority of theorems in commutative
algebra do not apply.

Despite their
non-noetherianness, 
perfectoid rings
have many nice properties, such as the
tilting equivalence (\cite{KedlayaLiu1} and \cite{ScholzePerfSpaces}).
Most if not all of these properties come from the
fact that perfectoid rings are \textit{geometrically meaningful},
which in this case precisely means
they are rings of functions associated to meaningful 
geometric objects 
via
Huber's theory of adic spaces (see \cite{HuberAdicSpaces},
the geometric objects in question are always
\textit{analytic} adic spaces). 
One may then ask the question

\begin{ques}
	Do perfectoid rings form a nice class of non-noetherian rings
	for which we can ``do commutative algebra''?
\end{ques}

This question is of special interest when the perfectoid rings
are (quotients of) a \textit{perfectoid Tate algebra}
over a field, i.e.
\[
	K\langle x_{1}^{1 / p^{\infty}}, \ldots, x_{n}^{1 / p^{\infty}}\rangle
.\] 
Such rings form the so-called perfectoid covers of
noetherian rigid analytic (adic) spaces,
and thus are the most closely related to noetherian rings.

The purpose of this work is to show, at least in one sense, the
answer to this question is negative.

\begin{thm}[Corollary \ref{cor:field:uncountable}]
	Let \(K\) be a perfectoid field. 
	Then the perfectoid Tate algebra in \(n\) variables
	has uncountable Krull dimension.
\end{thm}

In \cite{LangLudwig}, Lang and Ludwig show 
that Fontaine's period ring
\(\mathbb{A}_{\inf}\) has infinite Krull dimension.
Building on this, Du shows in \cite{DuUncountable}
that 
\(\mathbb{A}_{\inf}\) in fact
has uncountable Krull dimension.
Du's method additionally applies to the ring of formal
power series over a nondiscretely valued ring.
Du uses a family of valuations arising from mutiplicative
Gauss norms on \(\mathbb{A}_{\inf}\), 
explicitly calculating some Newton polygons of elements
and estimating the 
asymptotic behavior of their Legendre transforms.

In this work, we conceptualize Du's argument, 
defining a framework in which Du's argument 
applies, which includes more general rings 
than are treated in \cite{DuUncountable}.
We also provide a version of Du's proof which
uses basic properties of 
little omega notation (a variant of Big O)
instead of explicit estimates.
Then, we apply this formalism to 
the perfectoid Tate algebra over \(\mathcal{O}_{K}\),
the ring of integers of \(K\),
and check that the ideals which form the uncountable 
chain do not contain our selected pseudo-uniformizer
\(\pi\).

There are two important points--firstly,
the roles of \(x\) and \(\pi\)
are swapped from the ``standard'' way that
one thinks about Newton polygons.
This exchange of the arithmetic and geometric variables
is reminiscent of ideas from the world of
varieties over function fields,
and reflects the ``perfectoid outlook''
that \(p\) (or more generally, \(\pi\)) should 
be treated as ``just another variable''.

Secondly, unlike many proofs about perfectoid
rings, the result in mixed characteristic
does not follow from tilting to characteristic
\(p\)--in particular, as the tilting
map is not additive, there no guarantee that
the untilt of an arbitrary prime ideal should 
even be closed under addition.
Instead, we prove the theorem in 
each characteristic separately, 
in both cases passing to the spherical completion,
and using the explicit description of the
spherical completion described in \cite{PoonThesis}.
In fact, passing to the explicit description of the spherical 
completion happens to be 
just what is needed for formalizing for Du's argument, 
which roughly corresponds to the fact that in order
to use Newton polygon methods, one has to 
``tell apart the coefficients from the variables''.

Finally, we apply our same formalism to prove 
that 
\(A_{\infty}\), in the notation of the paper
\cite{PerfectoidSignature}, has uncountable
Krull dimension.
We also prove that a large 
(non-noetherian) subring of 
\(\widehat{R^{+}}\), the \(p\)-adic completion
of the absolute integral closure of a noetherian 
local domain \(R\) of mixed characteristic 
and dimension greater than 2,
has uncountable Krull dimension. 
This is related to a question posed in \cite{HeitmannRplus}.

All of the examples of rings with uncountable Krull 
dimension via Du's method have the one thing in common:
they are ``nondiscrete'' with respect to one variable,
and complete with respect to another. 
For example, if we consider \(R = \mathcal{O}_{K}[[x]]\), where
\(K\) is a perfectoid field,
\(R\) is complete with respect to \(x\), and nondiscrete
with respect to \(\pi\) (the pseudo-uniformizer).
On the other hand, if \(R = \tateone{K}\),
then \(R\) is complete with resepct to \(\pi\) and
nondiscrete with respect to \(x\) 
(it is also nondiscrete with respect to \(\pi\), 
but this is immaterial to the argument).
Roughly, these two facts guarantee the ability 
to construct Newton polygons of (sufficiently)
arbitrary shapes. 
Completeness allows for infinite Newton polygons and
ensures that a convergence condition is not a
restriction on the possible shapes of Newton polygons,
while nondiscreteness allows these Newton polygons
to approach 0 as the index variable goes to infinity.
Uncountability via Du's method follows from
the ability to construct such Newton polygons of
elements in \(R\).

Thus, 
we expect that any perfectoid ring which
``looks at least two dimensional''
(for example, has a regular element which is 
algebraically independent from \(\pi\))
to have uncountable Krull dimension.
Furthermore, we expect that Du's method
will apply to such rings, given
a sufficently explicit way of writing down
Newton polygons (we make this precise in 
Section \ref{sec:formal}).

In Section 2, we give a bit of technical background 
on Mal'cev-Neumann Series, the key ingredient in 
Poonen's explicit description of spherical completions.
Section 2.1 is devoted
to proving the multiplicativity of Gauss norms
on Mal'cev-Neumann series in a very general situation.
In this general setting, we give an example of an
element whose Gauss norm is not a maximum and is instead
an honest supremum. 
Getting around such examples is quite technical
and not strictly necessary for our examples, but 
it could have independent interest.
Section 2.1 can be skipped for a first reading.
In Section 3, we define the an abstract formalism for
Newton polygons of elements of a ring,
and show a few properties.
In Section 4, we outline Du's argument and prove it
for an arbitrary Newton polygon formalism.
In Section 5, we apply Du's theorem to various perfectoid
rings, including the (p-adic completion of) the
absolute integral closure and the perfectoid Tate algebra.

\subsection{Acknowledgements}

The author would like to thank Kiran Kedlaya for providing 
much support during the process of writing this paper. 
Moreover, the author would like to thank 
Hanlin Cai, Gabriel Dorfsman-Hopkins, 
Linquan Ma, Alexander Mathers, Shravan Patankar, Karl Schwede, 
Kevin Tucker, and Peter Wear for helpful conversations and comments. 
In addition, the author thanks 
Shravan Patankar for pointing out an error in an early draft of this paper.

This material is based on work partially supported by the 
National Science Foundation Graduate Research Fellowship Program under 
Grant No. 2038238, 
and a fellowship from the Sloan Foundation.

\section{Background on Mal'cev-Neumann Series}

We define Mal'cev-Neumann rings in a slightly more
general setting than they sometimes appear in the literature,
as we will need this later.

\begin{defn}
	Let \(R\) be a ring,
	and let \(\Gamma\) be a totally ordered group.
	The \textit{Mal'cev-Neumann series} (in one variable)
	over \(R\) with value group \(\Gamma\)
	is the set of formal sums
	\[
	f = \sum_{i \in \Gammanonneg}^{} a_{i}t^{i} 
	\] 
	such that the
	support of \(f\), 
	i.e. \(\{i | a_{i} \neq 0\}\), 
	is well-ordered.
	It is denoted \(R\malcevt\).
\end{defn}

One can define addition and multiplication
as with standard power series, using 
the well-ordered property to guarentee that
the ring is closed under multiplication.
Note that since we only allow positive exponents, 
\(R\malcevt\) is the ``integral'' version
of the Mal'cev-Neumann series.
We can analogously define a ``rational'' version,
which is a field when \(R = K\), and we denote
this by \(K\malcevtrat\).

\begin{defn}
	Let \(R\) be a valued ring.
	Then the 
	\textit{Gauss norm of radius \(\rho\) }
	on 
	\(R\malcevt\) is given by 
	\[
	\sup_{i} \{|a_{i}|\rho^{i}\}
	\] 
	for \(0 < \rho < 1\).
\end{defn}

Note that we have not yet proven that the Gauss norm
is a norm, we will do this in a bit. 
We also make use of the following
(slightly generalized) \(p\)-adic version 
of the Mal'cev-Neumann series, as
defined in \cite{PoonThesis}.

\begin{defn}
	Let \(R\) be a \(W(k)\)-algebra, for \(k\)
	a field of characteristic \(p\).
	Let \(\Gamma\) be a totally ordered group
	with a fixed embedding \(\mathbb{Z} \ins \Gamma\).

	We say a series 
	\(\sum_{i}^{} a_{i}t^{i} \) in \(R\malcevt\) is
	a \textit{null series} if
	\(\sum_{j \in i + \mathbb{Z}}^{} a_{j}p^{j} = 0 \) in \(R\).
	The set of null series is an ideal \(N\) 
	(use the proof of \cite{PoonThesis}, Prop 4.3),
	and we define 
	the \(p\)-adic Mal'cev-Neumann series as
	\[
	R\malcevp \colonequals R\malcevt / N
	\] 
\end{defn}

As before, we can define a ``rational''
version for a field \(K\) which contains
the fraction field of \(W(k)\), 
with similar operations.
We denote this by \(K\malcevprat\).
We record the following folklore lemma, due to 
Kedlaya:

\begin{lem}[Kedlaya]
	\label{lem:folklore:malcevp}
	Let \(k\) be a perfect field of characteristic \(p\).
	Let \(R\) be a \(p\)-complete \(W(k)\)-algebra
	which embeds into \(W(R / p)\).
	Then
	\[
		R\malcevp \isom \widehat{R\malcevt} / (t-p)
	\] 
	where the completion is \(p\)-adic.
\end{lem}

\begin{proof}
	First, consider the natural projection from 
	\(R\malcevt\) to \(R\malcevp\) from the definition
	of \(R\malcevp\).
	Now, take the induced map on \(p\)-adic completions.
	This map is surjective as completions preserve
	surjections,
	and as \(R\malcevp\) is \(p\)-adically complete, 
	we have 
	\[
	\pi : \widehat{R\malcevt} \to R\malcevp
	.\] 
	So it is enough to show that the kernel of this
	map is \((t-p)\).

	First, however, we show that the kernel is contained in 
	\((t,p)\).
	Consider the \(p\)-adic norm on \(R\) (restrict the \(p\)-adic
	norm from \(W(R / p)\)).
	Then let \(|\cdot|\) be the 
	Gauss norm on \(R\malcevt\) of radius 
	\(\epsilon = |p|\) which extends
	the \(p\)-adic norm on \(R\).
	Suppose \(f = \sum_{i}^{} a_{i}t^{i} \) 
	is a null series,
	but not in the ideal \((t,p)\).
	Then there exists an index \(i\) with 
	\(i < 1\) such that \(\epsilon < |a_{i}|\).
	Then, consider the sum
	\[
	f^{\prime} = \sum_{j \in i + \mathbb{N} \cup \{0\}}^{} a_{j}t^{j},
	\] 
	which is zero in \(R\) because \(f\) 
	is a null series.
	Thus, 
	\[
		a_{i}p^{i} = \sum_{j \in i + \mathbb{N}}^{} a_{j}t^{j} 
		\tag{*}
	.\] 
	However, the norm of the first term 
	of \(f^{\prime}\) has
	\[
		\epsilon^{1 + i} = \epsilon (\epsilon^{i}) = \epsilon |t^{i}| < |a_{i}t^{i}|
	\] 
	every other term has \(1 + i < j\), so 
	\[
	|a_{j}t^{j}| < |t^{j}| < \epsilon^{1 + i}
	.\] 
	By the strong triangle inequality, the left hand side
	of Equation (*).
	must have greater norm than the right hand side, a 
	contradiction.
	
	Finally, we show that \(f\) is indeed in \((t - p)\),
	though here we need to use the fact that we took the
	\(p\)-adic completion.
	Since \(\{t-p,p\}\) is a generating set for the ideal
	\((t,p)\), we may write
	\(f = a(t-p) + bp\). 
	We now show that \(b \in \Ker \pi\):
	as \(t-p \in \Ker \pi\), \(bp\) certainly is.
	Now,
	it is clear from the definition of null series
	that \(N\) is a \(p\)-saturated ideal, that is
	\(bp \in N \implies b \in N\).
	So \(b\in N\), and 
	\[
		f = a(t-p) + p(c(t-p) + dp) = a(t-p) + pc(t-p) + p^{2}d
	.\] 
	And we see that mod \(p^{2}\), \(f\) is divisible by \(t-p\).
	Continuing by induction, f is divisible by \(t-p\) mod \(p^{n}\)
	for all \(n\), so it is divisible by \(t-p\) in the completion.
\end{proof}

We also have the following lemma, which 
is a modified version of \cite[Proposition 4.4]{PoonThesis},
but adjusted to our context.

\begin{lem}
	\label{lem:malcevp:unique}
	Let \(R\) be a \(p\)-complete \(W(k)\)-algebra which embeds
	into \(W(R / p)\).
	Then 
	any element in 
	\(R\malcevp\) 
	can be written uniquely as a sum
	\(f = \sum_{i}^{} a_{i}p^{i} \)
	with \(a_{i} \in R\) such that \(p \nmid a_{i}\).
\end{lem}

\begin{proof}
	We follow \cite[Proposition 4.4]{PoonThesis}, in particular
	choose the set of representatives 
	\(\Gamma^{\prime} = \Gamma \cap [0,1]\) 
	of \(\Gamma / \mathbb{Z}\), and let 
	\(f = \sum_{i}^{} a_{i}p^{i}\).
	Then we can consider for any 
	\(\gamma \in \Gamma^{\prime}\),
	\(f_{\gamma} = \sum_{n \in \mathbb{N}}^{} a_{\gamma + n}p^{n}\),
	which is an element of \(\widehat{R}\).
	Since \(R\) embeds into \(W(R / p)\),
	we have that \(f_{\gamma}\) has a unique
	representation in \(\widehat{R}\).
	Then, finish the proof with exactly the
	same logic as \cite{PoonThesis}: 
	\[
	\sum_{\gamma \in \Gamma^{\prime}}^{} f_{\gamma} 
	\] 
	is an element of \(R\malcevp\) which is equivalent 
	to \(f\).
\end{proof}

We would like to define the \(p\)-adic version of the Gauss norm,
but we must be slightly careful about what the 
``norm on the base ring'' precisely is:

\begin{defn}
	\label{defn:normfamily:padic}
	Let \(k\) be a perfect field of characteristic \(p\),
	and
	let \(R\) be a \(W(k)\)-algebra which 
	embeds into \(W(R / p)\).
	We say that \(R\) \textit{has a family of \(p\)-adic norms}
	if 
	\(R\) has a (continuous) family of 
	norms/valuations \(|\cdot|_{\rho}^{\text{base}}\)
	for all \(\rho \in [0,1]\),
	subject to the conditions that
	\begin{itemize}
		\item \(|p|_{\rho}^{\text{base}} = \rho\)
		\item \(|\cdot|_{\rho}^{\text{base}}\) agree
			for any element which is not divisible
			by \(p\).
	\end{itemize}
\end{defn}

\begin{ex}
	Suppose that \(R / p\) is perfect and 
	that \(R\) can be embedded in \(W(R / p)\).
	Then \(R\) has a family of \(p\)-adic norms 
	by restricting from the Witt vectors.
	In this case, we can write any element \(f\) of \(R\) 
	uniquely as a sum \(f = \sum_{i}^{} a_{i}p^{i} \) 
	where \(a_{i}\) is in the image of the multiplicative
	lift in \(W(R / p)\).
\end{ex}

\begin{defn}
	Suppose \(R\) is a \(W(k)\)-algebra which has a family of 
	\(p\)-adic valuations
	and can be embedded into \(W(R / p)\).
	Then the \(p\)-adic Gauss norm of 
	radius \(\rho\) on  
	\(R\malcevp\) is
	defined to be
	\[
	|f|_{\rho} :=
	\sup_{i} \{|a_{i}|_{\rho}^{\text{base}}\rho^{i}\}
	.\] 
\end{defn}

One of the main applications of Mal'cev-Neumann rings
is constructing the spherical completion of a field
of characteristic \(p\).

\begin{thm}
	[Poonen, \cite{PoonThesis}, Corollary 5.5-5.6]
	Let \(K\) be a nonarchimedean
	algebraically closed field
	with valuation \(v\),
	residue field \(k\) and value group \(\Gamma\).
	Then there exists a unique 
	(up to non-unique isomorphism)
	spherically complete extension of \(K\),
	which is also maximally complete
	(\cite{KaplanskyMaximal1}, \cite{KaplanskyMaximal2}).
	We denote this as \(K^{sp}\).
	Being a field, \(K\) contains either 
	\(\mathbb{Q}\) or \(\mathbb{F}_{p}\).
	\begin{itemize}
		\item If the restriction of \(v\) 
			to \(\mathbb{Q}\) or 
			\(\mathbb{F}_{p}\) 
			is the trivial valuation,
			then \(K^{sp} \isom k\malcevtrat\).
		\item If \(K\) has characteristic zero
			and the restriction of \(v\) to 
			\(\mathbb{Q}\) is the \(p\)-adic 
			valuation, then
			\(K^{sp} \isom \Frac(W(k))\malcevprat\).
	\end{itemize}
\end{thm}

\begin{rmk} \label{rmk:piadic:malcev}
	If we have a \(W(k)\)-algebra \(R\) of mixed characteristic,
	we may choose a pseudo-uniformizer \(\pi\) 
	with \(\pi^{p} | p\)
	and choosing an element 
	\(f \in R\malcevt\)
	whose \(t\)-adic valuation is equal to \(\pi\),
	we can adapt Poonen's construction to 
	use a chosen pseudo-uniformizer,
	and our expansions will be \(\pi\)-adic 
	instead of \(p\)-adic.
	All of our proofs about \(p\)-adic Mal'cev-Neumann
	series apply to the \(\pi\)-adic variant
	with the proofs simply replacing \(p\) with \(\pi\).
\end{rmk}

\subsection{Gauss Norms on Mal'cev-Neumann Series}

We show that the Gauss norm is a norm, and that it is multiplicative.
We will prove multiplicativity 
for the full ring of Mal'cev-Neumann series. 
However, most concrete applications require multiplicativity
only for various subrings with finiteness conditions,
whose proofs are much easier.
We will do the general case, skipping
the simpler ad hoc cases. This will require a bit of extra
work.

We now investigate an obstacle to
proving the multiplicativity of 
Gauss norms on Mal'cev-Neumann rings
(and working with these rings in general).
For many proofs of multiplicativity of Gauss norms
on simpler rings,
one uses that the \(\sup\) in
the definition of the Gauss norm is a maximum.
However, this is not true in general for 
Mal'cev-Neumann rings.

\begin{defn}
	The \textit{Puiseux series} over the ring \(R\) 
	is the power series ring \(R[[x]]\) adjoined with 
	all \(n\)-th roots of \(x\), that is 
	\[
		\bigcup_{n \in \mathbb{N}}^{} R[[x^{1 / n}]] 
	.\] 
	We will denote this ring by \(R\puiseuxx\).
\end{defn}

\begin{ex} \label{ex:gauss:sup}
	Let \(R = k\puiseuxx\).
	Fix \(\mu\) between zero and one, and 
	equip \(R\) with the valuation
	associated to the Gauss norm of radius \(\mu\).
	Now, we will construct an element of 
	\(R\puiseuxx\malcevt\) 
	whose Gauss norm is not a maximum. 
	Pick \(\rho\) between zero and one.
	Pick a sequence \(\delta_{n} \to 0\) with
	\(\delta_{n} \in \Gamma \cap [0,1]\)
	such that 
	\[
		\delta_{n} < \frac{1}{n} \log_{\rho}(\mu)
	.\] 
	Then, consider the element
	\[
	f = \sum_{n=1}^{\infty} x^{1 + \frac{1}{n}} t^{2 - \delta_{n}} 
	.\] 
	Note that the support of \(f\), the set \(\{2 - \delta_{n}\}\),
	is well-ordered.
	We have
	\[
	|f|_{\rho} = \sup \{\mu^{1 + \frac{1}{n}} \rho^{2 - \delta_{n}}\}
	= \mu \rho^{2}
	,\] 
	yet the norm of every term of \(f\) is less than 
	\(\mu\rho^{2}\).
\end{ex}

Example \ref{ex:gauss:sup} is not necessarily an obstacle
to the Gauss (semi-)norms on \(R\malcevt\) and \(R\malcevp\) 
being multiplicative.
However, it is an obstacle to the standard proofs of
multiplicativity.
Later in the proof, we will see that the real difficulty
is not the fact that the norm is not achieved,
but that there are ``coefficients with arbitrarily close norm''
to the norm of \(f\).
To reason about such elements, we introduce the following
notion.

For the following, let \(f = \sum_{i}^{} a_{i}r^{i} \) 
be an element of the \(t\)- or 
\(p\)-adic Mal'cev-Neumann series.
If one wants, one may only think about the \(p\)-adic series, the
proofs for \(t\)-adic are strictly easier.

\begin{defn}
	The \(\argnorm\) of an element \(f\) is the smallest 
	value \(i\) for which one of the following happens.
	\begin{itemize}
		\item The Gauss norm of \(f\) is achieved at index \(i\). 
		\item There exists a sequence \(i_{n} \to i\) 
			which achieves the Gauss norm of \(i\).
	\end{itemize}
	Such a smallest \(i\) exists by the well-ordered property.
\end{defn}

In addition, we fix a \(\rho\) 
and fix the following notation:
let \(\restrint{I} f\) be the 
element with coefficients
of \(f\) supported on \(I\).
Let \(\restr{I}{N} f\) be the element
whose coefficients correspond to those
of \(f\) such that \(i \in I\) and 
\(N \leq |a_{i}|\rho^{i}\).

We record two lemmas
\begin{lem}
	\label{lem:malcevelts:box}
	[Box Lemma]
	There exist numbers \(\epsilon_{a}, \delta_{a}\), 
	depending only on \(f\), for which 
	\[
	\restr{I}{|f| - \delta_{a}} f = 0
	\] 
	for \(I = (\argnorm f, \argnorm f + \epsilon_{a})\).
\end{lem}

\begin{proof}
	If there did not exist such numbers,
	we would have that for all \(\epsilon\) 
	greater than \(0\), there exists a nonzero
	coefficient (i.e. an \(i\) for which \(a_{i} \neq 0\)),
	between \(\argnorm f\) and \(\argnorm f + \epsilon\),
	which contradicts well-orderedness of the support
	of \(f\).
\end{proof}

\begin{lem}
	\label{lem:malcevelts:bar}
	[Bar Lemma]
	Given an \(\epsilon\) small enough, there exists a 
	\(\delta_{b,\epsilon}\) for which 
	\[
		\restr{I}{|f| - \delta_{b}} f = 0
	\] 
	for \(I = [0, \argnorm f - \delta_{b, \epsilon})\).
	If \(\epsilon\) is a fixed \(\epsilon_{b}\),
	then we shorten the notation to
	\(\delta_{b} := \delta_{b,\epsilon_{b}}\).
\end{lem}

\begin{proof}
	If the negation of this lemma holds, then one can find a sequence
	\(i_{n}\) such that \(i_{n} < \argnorm f - \epsilon\) for all \(n\),
	but \(\sup { |a_{i}| \rho^{i} } = |f|\).
	This contradicts the fact that \(\argnorm f\) is the argument
	of the norm (e.g. pass to a convergent subsequence).
\end{proof}

\begin{figure}[h]
\centering
\includegraphics[width=0.5\textwidth]{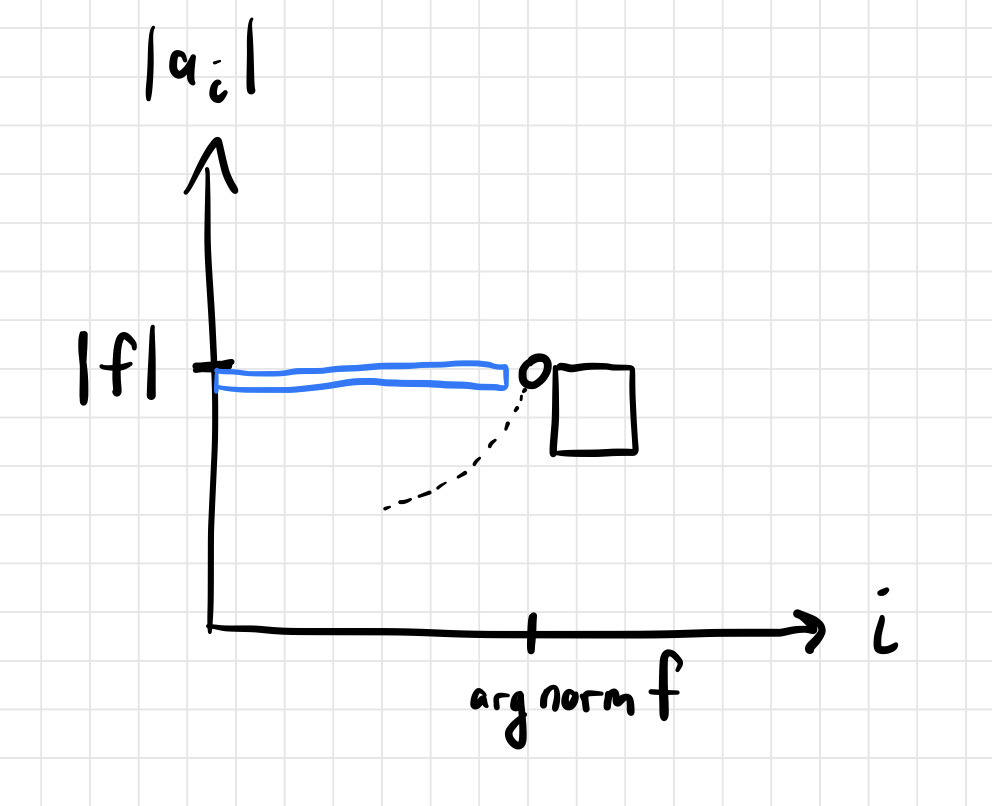}
\caption{A norm graph of an element in the Mal'cev-Neumann ring
with Gauss norm not achieved at a maximum.
There are no coefficients in the black box by the 
Box Lemma (\ref{lem:malcevelts:box}) and there are
no coefficients in the blue rectangle by the Bar Lemma 
(\ref{lem:malcevelts:bar}).}
\end{figure}

\begin{hyp}
	For the rest of the section, assume \(R\) is \(p\)-complete
	if we are in the \(p\)-adic case
	(i.e. \(f \in R\malcevp\)).
	This is used for the invocation of 
	Lemmas \ref{lem:folklore:malcevp} and \ref{lem:malcevp:unique}.
\end{hyp}

Our proofs of the properties of the Gauss norm in the
\(t\)-adic case rest on using properties of the standard
formula for a product
\[
\left( \sum_{i}^{} a_{i}t^{i}  \right)
\left( \sum_{j}^{} b_{j}t^{j}  \right)
= \sum_{k}^{} \sum_{i+j=k}^{} a_{i}b_{j} t^{k}  
.\] 
However, for the \(p\)-adic case we don't have this
formula, as we must account for ``\(p\)-adic carrying''.
Thus, we use the following slightly weaker property
which applies in the \(p\)-adic case (and of 
course still holds for the \(t\)-adic 
Mal'cev-Neumann series).

\begin{lem}
	\label{lem:padic:elts}
	Let 
	\[
	f = \sum_{i}^{} a_{i}r^{i}, 
	g = \sum_{j}^{} b_{j}r^{j} 
	\] 
	be elements of the 
	\(t\)-adic or \(p\)-adic Mal'cev-Neumann series
	(\(R\malcevt\) or \(R\malcevp\)),
	with \(r = t\) or \(p\) respectively.
	Then we have the following property:

	\[
	fg = \sum_{k}^{} \sum_{k^{\prime} \leq k}^{} 
	\left( \sum_{i + j = k, (i,j) \in S_{f,g,k^{\prime}}}^{} a_{i}b_{j} r^{k}  \right)
	,\] 
	where \(S_{f,g,k^{\prime}}\) is a finite subset of 
	\(\Gammanonneg \times \Gammanonneg\)
	such that 
	\(\bigcup_{k^{\prime} \leq k}^{} S_{f,g,k^{\prime}} \) 
	is finite.

	Roughly, \(S_{f,g,k^{\prime}}\) is the 
	set of indices that, after \(r\)-adic carrying 
	has been done, end up in the \(r^{k}\)-th place.
\end{lem}

\begin{proof}
	For \(t\)-adic series, just take 
	\(S_{f,g,k^{\prime}}\) to be zero 
	for \(k^{\prime} \leq k\) and use
	the standard formula for products
	stated above.

	For \(p\)-adic series, we can use
	Lemma \ref{lem:folklore:malcevp}
	to see that we may first 
	multiply elements as \(t\)-adic series, and
	then quotient by \(t-p\).
	we can ``collect in terms of \(p\)''
	(by the uniqueness of the representation in 
	\(R\malcevp\), Lemma \ref{lem:malcevp:unique}).
	As the quotient is canonically isomorphic to Poonen's
	definition, the
	elments \(f\) and \(g\) must have well-ordered support.
	However, all possible \(p\)-adic carry operations
	happen in the set \(k\mathbb{Z}\),
	and thus \(k\mathbb{Z} \cap \sprt(f)\) must
	have a minimum element, and respectively for \(g\).
	Thus, the sum over
	all the \(S_{f,g,k^{\prime}}\) must be finite.
\end{proof}

\begin{lem}
	\label{lem:malcev:prodsupp}
	Let \(f, g\) be two elements in 
	\(R\malcevt\) or \(R\malcevp\).
	Then \(\supp fg\) is contained
	in \((\supp f + \supp g)\mathbb{N}\).
\end{lem}

\begin{proof}
	If \(f\) and \(g\) are in \(R\malcevt\),
	then by the product identity the support
	of \(fg\) is in \(\supp f + \supp g\).
	On the other hand if they are in \(R\malcevp\),
	\(p\)-adic carry operations may 
	cause an index in the support to increase by 
	an integer.
\end{proof}

For the following, fix \(\rho\) with \(0 < \rho \leq 1\).

\begin{lem} \label{lem:malcevt:tri}
	Let \(R\) be a valued ring.
	Then the Gauss norm on \(R\malcevt\) 
	(resp. \(R\malcevp\)) satisfies the 
	strong triangle inequality.
\end{lem}

\begin{proof}
	In the \(t\)-adic case, this follows from
	the formula for the sum of two elements
	with the strong triangle inequality 
	on \(R\).

	In the \(p\)-adic case, by 
	Lemma \ref{lem:folklore:malcevp}
	we may first do the addition using the
	simpler \(t\)-adic formula,
	and then do ``\(p\)-adic carrying''.
	However, any \(p\)-adic carry 
	cannot increase the norm as \(\rho \leq 1\).
\end{proof}

\begin{cor} \label{cor:malcevt:submult}
	The Gauss norm on \(R\malcevt\) 
	(resp. \(R\malcevp\)) is 
	submultiplicative.
\end{cor}

\begin{proof}
	Using the above formulas for the 
	product of two elements 
	\(f = \sum_{i}^{} a_{i}r^{i} \) 
	and \(g = \sum_{j}^{} b_{j}r^{j} \)
	we see that each coordinate is a sum
	of elements of the form \(a_{i}b_{j}r^{i+j+n}\),
	where \(n \neq 0\) only if there is a 
	\(p\)-adic carry operation.
	Now, use the strong triangle inequality
	and the fact that \(\rho \leq 1\)
	to conclude that the norm of 
	\(fg\) must be less than or equal to the greatest of
	these coordinates, but
	by submultiplicativity on \(R\) this norm
	is less than or equal to \(|f||g|\).
\end{proof}

\begin{lem} \label{lem:malcevt:cont}
	The Gauss norms on
	\(R\malcevp\)
	for \(0 < \rho \leq 1\) 
	form a continuous family.
\end{lem}

\begin{proof}
	The function \(\rho \mapsto |a_{i}|\rho^{i}\) is
	continuous, and 
	a sup of continuous functions is continuous.
	Note here that
	in the \(p\)-adic setting,
	\(|a_{i}| = |a _{i}|^{base}_{\rho}\) 
	as in Definition \ref{defn:normfamily:padic}.
	Since we may take \(a_{i}\) 
	such that \(p \nmid a_{i}\),
	\(|a_{i}|\) does not depend on
	the norm which is on \(R\).
\end{proof}

To show that the Gauss norm is multiplicative,
we need one key lemma as an input. 
Essentially, the idea is that one may reduce to
the case where all elements \(f\) are supported 
in a neighborhood of the \(\argnorm\).

\begin{lem}
	\label{lem:only:argnorm}
	Let \(f, g\) be in \(R\malcevt\) or
	\(R\malcevp\), and let 
	\(|\cdot|\) be the Gauss norm
	of radius \(\rho\).
	Let \(I^{h}_{\epsilon} = 
	(\argnorm h - \epsilon, \argnorm h]\) 
	for \(h \in R\malcevt\) (resp. \(R\malcevp\)).
	If the \(h\) is clear from context, we
	will omit it.

	Then there exists an \(\epsilon_{1}, \epsilon_{2}\),
	and a \(\delta\)
	greater than zero 
	with the property
	that 
	\[
	\left| \restr{I_{\epsilon_{1}}}{|f|-\delta} f\right|
	\left| \restr{I_{\epsilon_{2}}}{|g|-\delta} g\right|
	= \left|\left( \restr{I_{\epsilon_{1}}}{|f|-\delta} f \right)
	\left( \restr{I_{\epsilon_{2}}}{|g|-\delta} g \right) \right|
	\implies 
	|f||g| = |fg|
	.\] 
\end{lem}

\begin{figure}[h]
\centering
\includegraphics[width=0.5\textwidth]{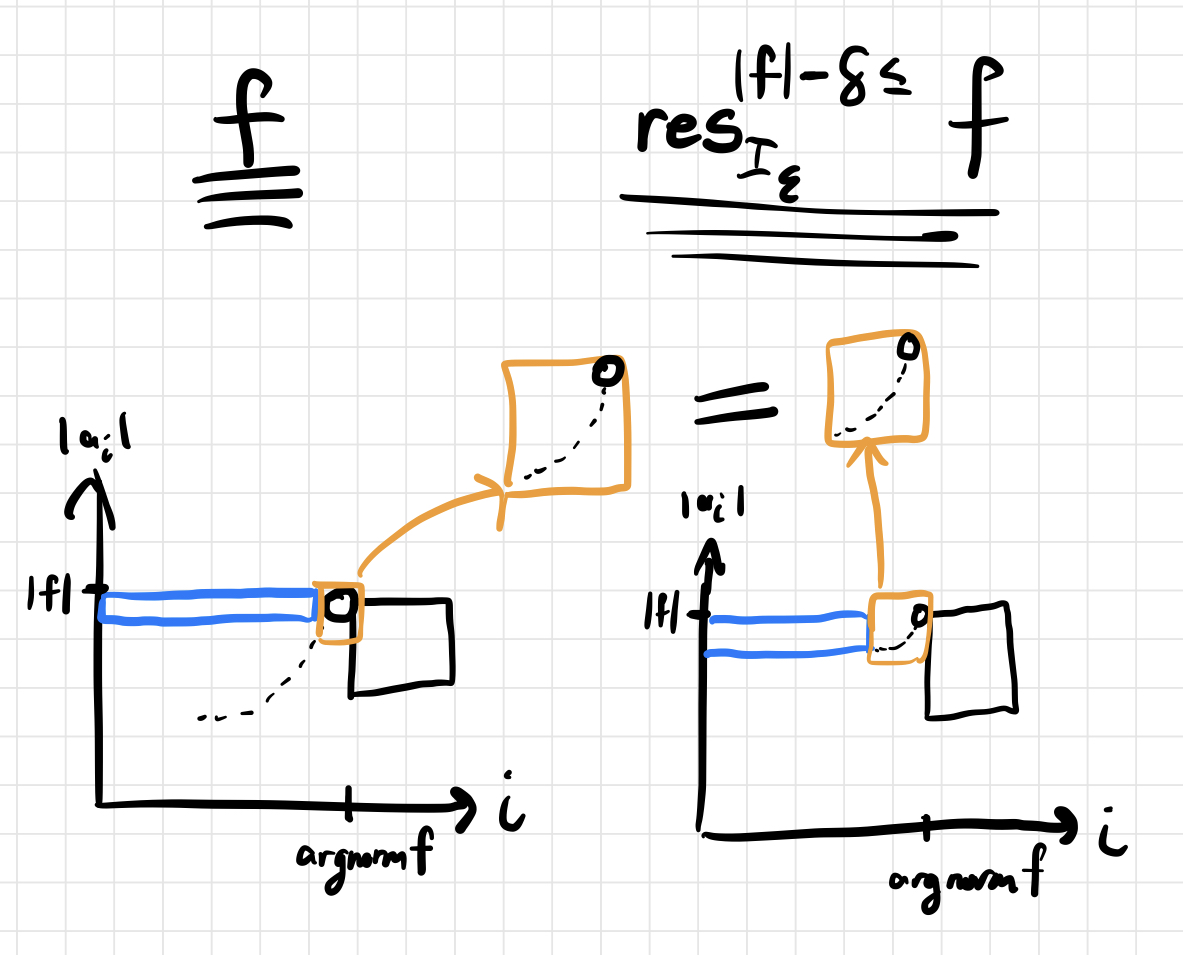}
\caption{A picture of \(\restr{I_{\epsilon}}{|f|-\delta} f\).
One should think of removing all of the coefficients of \(f\) 
except those which attain a norm which is very close to the norm
of \(f\). 
If the set of norms of terms of \(f\), (i.e. \(\{|a_{i}|\rho^{i}\}\))
does not accumulate at \(|f|\) as in
Example \ref{ex:gauss:sup}, then for sufficiently small 
\(\epsilon\) and \(\delta\),
\(\restr{I_{\epsilon}}{|f|-\delta} f\) has at most one term,
i.e. the orange box in the picture
would have only a single filled-in dot which attains the norm,
and the classical proof works.
Lemma \ref{lem:only:argnorm}
shows that to show multiplicativity, while one cannot
reduce to a single point as in the classical proof, one
can reduce to the (arbitrarily small) orange box of 
width \(\epsilon\) and height \(\delta\).}
\end{figure}

\begin{proof}
	We will show that 
	every coefficient of 
	\(\restr{I_{\epsilon^{\prime}}}{|f||g|-\delta} fg\)
	is a sum of products of coefficients
	of \(\restrint{I_{\epsilon_{1}}} f\)
	and \(\restrint{I_{\epsilon_{2}}} g\)
	for some \(\epsilon^{\prime}\),
	from which the claim follows.

	By Lemma \ref{lem:padic:elts},
	any such coefficient of \(fg\) 
	is a sum of products of coefficients
	\(a_{i}\) and \(b_{j}\) such
	that \(i + j \leq \argnorm f + \argnorm g\).

	Choose \(\ebox_{f}\) and \(\ebox_{g}\) 
	as in the Box Lemma (Lemma \ref{lem:malcevelts:box}).
	Then, choose \(\ebaar_{f} < \ebox_{g}\) 
	and \(\ebaar_{g} < \ebox_{f}\)
	as in the Bar Lemma (Lemma \ref{lem:malcevelts:bar}).
	From these lemmas, we have corresponding
	\(\dbox_{f}, \dbaar_{f}, \dbox_{g}, \dbaar_{g}\).
	Finally, choose \(\delta\) so that
	\(\delta < 
	\min \{|f| \dbox_{f}, |f| \dbaar_{f},
	|g| \dbox_{g}, |g| \dbaar_{g}\}\).
	Then, for any product of coefficients
	of \(f\) and \(g\) 
	in 
	\(\restr{I_{\epsilon^{\prime}}}{|f||g|-\delta} fg\),
	we have
	\[
		|f| - \frac{\delta}{|g|} < |a_{i}|\frac{|b_{i}|}{|g|}
		< |a_{i}|
	,\] 
	which gives by the Bar Lemma that
	\[
	\argnorm f - \epsilon_{1} < i
	\implies
	j < \argnorm g + \epsilon_{1}
	,\] 
	using that \(i + j \leq \argnorm f + \argnorm g\).
	Now, by the Box Lemma,
	\[
	j < \argnorm g + \epsilon_{1}
	\implies
	j < \argnorm g
	.\] 
	We can use a completely symmetric argument to deduce
	that \(\argnorm g - \epsilon_{2} < j\) 
	and \(i < \argnorm f\),
	and these four statements taken together 
	are what we wanted to show.
\end{proof}

\begin{rmk}
	The previous proof shows that we may also choose
	our 
	\(\epsilon\)'s less than \(\epsilon_{1}\) and
	\(\epsilon_{2}\) and the statement still applies.
\end{rmk}

\begin{cor}
	\label{cor:malcevp:mult}
	The Gauss norms on 
	\(R\malcevt\) and
	\(R\malcevp\)
	are multiplicative.
\end{cor}

\begin{proof}
	Let \(f, g\) be elements
	of \(R\malcevt\) or
	\(R\malcevp\).
	Assume for the sake of 
	contradiction that 
	\(|fg| < |f||g|\) strictly.
	Pick \(\epsilon_{1}\),
	\(\epsilon_{2}\) as
	in the previous lemma 
	so that 
	\(|fg| < (|f| - \delta) (|g| - \delta)\),
	where \(\delta\) is given
	by the previous lemma.
	We can assume that
	\(f = \restr{I_{\epsilon_{1}}}{|f|-\delta} f\)
	and likewise for \(g\).
	Then there exists a 
	unique smallest element \(i\) of
	\(\supp f\)
	by well-orderedness.
	Likewise, we have a smallest index
	\(j\) for \(g\) in \(I_{\epsilon_{2}}\).
	Then, 
	\(a_{i}b_{j}\) must be a coefficient
	in the product, and there
	can be no cancellation 
	by minimality.
	But then,
	\(|fg| < |a_{i}||b_{i}| = |a_{i}b_{i}|\),
	which is a contradiction 
	by multiplicativity downstairs.
\end{proof}

\section{Formalisms} \label{sec:formal}


We first define a notion which axiomatizes
the idea of a ``collection of valuations'':

\begin{defn}
	Let \(R\) be a ring, and let 
	\((C,\oplus,\otimes,\leq)\) be a
	partially ordered tropical
	semiring. 
	Then a map of sets 
	\(\LN : R \to C \) is valuative if
	\begin{enumerate}[(i)]
		\item \(\LN\) is superadditive,
			that is \(\LN(a) \oplus \LN(b) \leq \LN(a + b)\)
		\item \(\LN\) is multiplicative,
			or \(\LN(a) \otimes \LN(b) = \LN(ab)\)
	\end{enumerate}
	for all \(a, b \in R\).
\end{defn}

As a first example, taking \(\LN\) to be a valuation
\(v\) of \(R\) on a totally ordered group \(\Gamma\), 
we see that 
\(v\) is valuative
(here, tropical addition is taking the \(\min\)). 
Also, if one takes any two valuations \(v, v^{\prime}\),
the natural map \(R \to \Gamma_{v} \times \Gamma_{v^{\prime}}\) 
is valuative.
We will cheifly be concerned with the case when
our tropical semiring \(C\) is 
\(\contposreal\), and our map 
``comes from a family of Gauss norms/valuations''.

%
%
%
%


For the rest of this paper, the Legendre
Transform always refers to 
the infimum Legendre transform
of a decreasing convex function \(F\) 
on the positive real numbers,
that is
\[
	\leg(F)(t) = \inf_{x} \{F(x) + xt\}
.\] 

\begin{defn}
	\label{def:npf}
	Let \(R\) be a ring, 
	and let \(U \ins \nonneg\) be an interval.
    A \textit{Newton polygon formalism}
	with domain \(U\) on \(R\) 
    is a diagram of sets 
    \[   
    \begin{tikzcd}
    R \arrow{r}{\newt} \arrow{rd}{\LN} & 
    \contposreal \arrow{d}{\leg} \\
    & \contposreal
    \end{tikzcd}
    \]
    where \(\leg\) is the Legendre transform
	(restricted to \(\contu\))
    such that 
    \begin{itemize}
        \item \(\LN\) is superadditive and multiplicative
        \item \(\newt(f)\) is convex for every \(f \in R\).
    \end{itemize}

	\(\newt(f)\) is called the \textit{Newton polygon}
	of \(f\) under the formalism.
\end{defn}

Note that the Newton polygon of \(f\) in the
definition has no requirement to be a polygon!
We refrain from calling it the ``Newton lower convex hull''
for consistency and recognizability.

Examples of Newton polygon formalisms include
standard examples of rings which have Newton polygons,
for example take \(R = K[x], R = K\langle x\rangle,\) or \(R = K[[x]]\) 
for \(K\) a valued field.
The main theorems in this paper will follow from the
existence of ``nonstandard'' Newton polygon formalisms
on various perfectoid rings \(R\).
However, all of the rings we consider are alike in that
the Newton polygon formalism 
``comes from a continuous 
family of multiplicative Gauss norms''.
The multiplicative norms correspond to valuations,
and the continuous family of valuations becomes
the map \(\LN\).

\begin{rmk}
	Since \(\LN\) must be the Legendre transform
	of \(\newt\), the data of a Newton polygon formalism
	is only the map \(\newt\).
\end{rmk}

\begin{rmk}
	Definition \ref{def:npf} implies that \(\newt(f)\) is
	invariant under multiplication by units.
	Indeed, if any unit has a nonzero Newton polygon,
	by multiplicativity its inverse should be
	negative, which is disallowed.
	In particular, a field may only have the trivial Newton
	polygon formalism.
\end{rmk}

Below, we give a technical lemma which shows that 
Mal'cev-Neumann rings have Newton polygon 
formalisms in the presense of reasonable 
Gauss norms.

\begin{prop} \label{prop:constr:npf}

	Let \(R\) be \(V\malcevp\) for \(V\) a 
	ring with a family of \(p\)-adic valuations 
	such that \(V\) embeds into 
	\(W(V / p)\).
	Respectively, assume \(R = V\malcevt\) 
	for \(V\) some valued ring.

	Suppose that the Gauss norms 
	on \(R\) of radius \(s\) are multiplicative
	norms for
	\(s \in U\).
	Then
	the associated
	valuations
	\(v_{s}: R \to \mathbb{R}\) 
	give a Newton polygon formalism
	with 
	\(\LN(f) = s \mapsto v_{s}(f)\) 
\end{prop}

\begin{proof}
	Let \(r = p\) or \(t\), respectively.

	We need to define a map \(\newt\),
	then we need to show that 
	for any \(f = \sum_{i}^{} a_{i}r^{i}\),
	\(\LN(f)\) has
	the above form, and further that \(\LN(f)\) 
	is a continuous valuative map.

	We define \(\newt(f)\) to 
	by its graph: let \(\Gamma(\newt(f))\)be the 
	nonincreasing lower convex hull 
	of the points \((i,v(a_{i}))\) in \(\mathbb{R}^{2}\).
	Thus \(\newt(f)\) is a continuous positive function,
	and so lies in \(\contposreal\).

	We can write \(\LN(f)\), the legendre
	transform of \(\newt(f)\), as
	\[
	s \mapsto \inf_{x \in \nonneg}\{\newt(f)(x) + sx\}.
	\] 
	On the other hand, 
	the valuation associated to the Gauss norm of
	radius \(s\) is 
	\[
		\inf_{i \in I}\{v(a_{i}) + s i\}
	,\] 
	where \(I\) is the set of nonzero indices of \(f\),
	i.e. \(I = \{i | a_{i} \neq 0\}\).
	We will show that in \(\LN(f)\), the infimum can
	be taken over \(I\).
	Then the two expressions match.

	We see that \(\newt(f)\) has the following property.
	For every \(x \in \nonneg\), there exists a neighborhood
	\(U_{x}\) of \(x\) with either
	\begin{enumerate}[(1)]
		\item \(\newt(f)\) is piecewise-linear on \(U_{x}\),
			or
		\item The set \(S^{\prime} \colonequals
			\{i | (i,v(a_{i})) \in \Gamma(\newt(f))\}\)
			accumulates to \(x\).
	\end{enumerate}
	Suppose \(x\) is the point where the infimum
	is achieved in \(\LN(f)(s)\) (this point must be
	finite, as \(\newt(f)\) is nonincreasing
	and positive and so will be dominated by the linear
	term eventually).
	Then, in case (1), such a minimum can only
	be attained at a node of a piecewise-linear function,
	so we conclude that \(x\) is a node, and that
	\(x \in I\).
	In case (2), we choose an accumulating set
	\(x_{j} \to x\) with \(x_{j} \in I\),
	and by continuity the limit of this sequence must
	be the infimum. 
	Thus, we may take the inf over \(x \in I\), as desired.

	Finally, the map
	\(\LN\) is valuative because
	\(v_{s}\) are multiplicative valuations
	(as the corresponding Gauss norms are multiplicative),
	and form a continuous family.
\end{proof}

\begin{ex}
	Let \(R = \mathbb{A}_{\inf}\).
	Let \(r = p\),
	and let \(S\) be the image
	of the multiplicative lift of \(\mathcal{O}_{K}^{\flat} \).
	Then the
	Gauss norms are multiplicative
	in this situation \cite[Chapter 3]{FFCourbes}.
	By the same argument as Proposition \ref{prop:constr:npf},
	there is a 
	Newton polygon formalism on \(\mathbb{A}_{\inf}\).
	These are 
	the Newton polygons on \(\mathbb{A}_{\inf}\) 
	used in \cite{DuUncountable}
	and \cite{FFCourbes}.
\end{ex}

We also record a few basic lemmas:

\begin{lem} \label{lem:npf:injection}
	Let \(R \inj R^{\prime}\) be an inclusion,
	and let \(\newt\) be a Newton polygon formalism
	on \(R^{\prime}\).
	Then \(\newt\) is a Newton polygon formalism on \(R\).
\end{lem}

\begin{proof}
	Straightforward verification from definitions.
\end{proof}

\begin{lem} \label{lem:npf:completion}
	Let \(K\) be a nonarchimedian field, and let 
	\(r\) be a pseudouniformizer.
	Let \(R\) be 
	\(\mathcal{O}_{K}[x],\)
	\(\mathcal{O}_{K}[[x]],\)
	\(\mathcal{O}_{K}\coperfx,\)
	\(\mathcal{O}_{K}\coperfxpow \) or
	\(\mathcal{O}_{K}\puiseuxx\).

	Assume \(\newt\) is a Newton polygon formalism
	on \(R\)
	coming from multiplicative
	\(x\)-adic Gauss norms
	(Proposition \ref{prop:constr:npf}),
	then this can be
	extended to a Newton polygon formalism on 
	\(\widehat{R}\) for the \(r\)-adic topology
	on \(R\).
\end{lem}

\begin{proof}


	In these cases, we actually have a
	ring map 
	\(\widehat{R} \inj \widehat{R}^{(r,x)}\)
	from the \(r\)-adic completion to the 
	\((r,x)\)-adic completion.
	We may see this by unwinding the
	limit definition of \((r,x)\)-adic completion. 
	In the \((r,x)\)-adic completion, 
	the \(x\)-adic Gauss norms are 
	a continuous family of (continuous) valuations,
	giving a Newton polygon formalism
	which restricts to the original one on \(R\).
	Now, by Lemma \ref{lem:npf:injection}
	we may deduce that they also give a Newton
	Polygon Formalism on \(\widehat{R}\) which
	restricts to the original Newton
	polygon formalism on \(R\).

\end{proof}

\section{Du's Argument} \label{sec:du:argument}

In this section, we will prove Theorem \ref{thm:du:uncountable},
using the properties of big \(O\) notation and its variants from
computer science. 
For completeness, we define it here.

\begin{defn}
	Let \(F\) and \(G\) be functions \(U \to \nonneg\) for
	some subset \(U \ins \nonneg\).
	We consider \(\lim \frac{F}{G}\), where all limits
	are taken as \(t \to a \in [0, \infty]\) such
	that \(U\) accumulates to \(a\).
	We say that
	\begin{itemize}
		\item \(F \in O(G)\) if \(\lim \frac{F}{G} < \infty\)
		\item \(F \in O^{\sup}(G)\) if \(\lim \sup \frac{F}{G} < \infty\)
		\item \(F \in \omega(G)\) if \(\lim \frac{F}{G} = \infty\)
		\item \(F \in \omega^{\sup}(G)\) if \(\lim\sup \frac{F}{G} = \infty\)
		\item \(F \in o(G)\) if \(\lim \frac{F}{G} = 0\)
		\item \(F \sim G\) if \(\lim \frac{F}{G} = 1\)
	\end{itemize}
	If there is ambiguity, we may add the ``approaches''
	into the notation, i.e. \(\omega_{t \to 0}(G)\).
\end{defn}

We remind the reader that \(O(G)\) and \(O^{\sup}(G)\) 
(resp. \(\omega(G)\) and \(\omega^{\sup}(G)\)) are not the same,
for example take a function which oscillates between
\(y = 1\) and \(y = x\) as \(x \to \infty\).
For clarity, we always be precise with which variant we are using.

\begin{lem}
	\(O(G)\) and \(O^{\sup}(G)\) are closed under addition.
\end{lem}

\begin{proof}
	Unwind definitions.
\end{proof}

\begin{defn}
	
	Let \(F\) and \(G\) as above.
	If there exists \(H \in o(G)\) such that
	\(|F - G| < H\) for all \(x\) sufficiently
	close to \(a\), then we say that
	\(F\) and \(G\) are 
	\textit{asymptotically arbitrarily close}.
\end{defn}

\begin{lem}
	\label{lem:aac:asymeq}
	If \(F\) and \(G\) are asmptotically
	arbitrarily close,
	then \(F \sim G\).
\end{lem}

\begin{proof}
	Unwind definitions.
\end{proof}

%
%

In order to state Du's Theorem, we will
now switch gears, back to Newton polygon formalisms.
In much of Newton polygon theory, theorems are deduced
by analyzing individual Newton polygons.
Our approach will be slightly different,
we aim to deduce facts about the ring \(R\) from the
\textit{possible} Newton polygons from a given formalism.

\begin{defn}
	The \textit{discrete approximation} of a function 
	\(G \in \contposreal\) is the piecewise-linear
	function with nodes
	\((n,G(n))\) for \(n \in \mathbb{N}\).
\end{defn}

\begin{defn}
	\label{def:discrete:approx}
	A Newton polygon formalism \(\newt\) 
	\textit{discretely approximates}
	a function \(H \in \contposreal\) 
	if there exists an element \(r \in R\) 
	whose Newton polygon \(\newt(r)\) 
	has
	\begin{itemize}
		\item \(\newt(r)\) and \(G\) 
			are asymptotically arbitrarily close in the sense
			of Lemma \ref{lem:aac:asymeq},
			that is there exists some \(H \in o(G)\) for which
			\(|\newt(r) - G| < H\).
		\item The slopes of secant lines of \(\newt(r)\) and \(G\) 
			are asymptotically arbitrarily close, that is 
			given \(x \in \nonneg\), for any
			\(x_{1}, x_{2} \in (x-\delta, x+\delta)\) for some
			small \(\delta\), we have that 
			\(|m_{G} - m_{\newt(r)}| < h^{\prime} \in o(G)\) 
			where \(m_{G}\) is the slope of the secant
			line connecting \((x_{1}, G(x_{1}))\) and
			\((x_{2}, G(x_{2}))\), and likewise for \(m_{\newt(r)}\).
	\end{itemize}
	where \(G\) is the discrete approximation
	of the function \(H\).
	
\end{defn}

\begin{rmk}
	Let \(G\) be a function whose discrete approximation
	is nonconvex.
	Then it is impossible for the standard Newton 
	polygon formalism (say, on \(V[[x]]\) for 
	a valued ring \(V\))
	to discretely approximate \(G\).
\end{rmk}

Next, we have enough to state Du's main theorem (\cite{DuUncountable}),
in our conceptualized form.

\begin{thm}[after Du] \label{thm:du:uncountable}
	Let \(R\) be a ring and \(\newt\) a Newton 
	polygon formalism with domain \((0,1]\) on \(R\).
	If \(\newt\) discretely approximates
	any (convex) function \(G \in \contposreal\) which has
	\(\lim_{i \to \infty} g(i) = 0\),
	then \(R\) 
	has uncountable Krull dimension.
\end{thm}

\subsection{Proof of Du's Theorem}

\begin{prop} \label{prop:leg:asymp}
	Let \(F\) and \(G\) be asymptotically arbitrarily close,
	in the sense of the previous section.
	Suppose further that the 
	slopes of secant lines of \(F\) and \(G\) are 
	asymptotically arbitrarily close as in 
	Definition \ref{def:discrete:approx}.
	Then we have 
	\(\leg(F) \sim \leg(G)\) as \(t \to 0\),
	where \(\leg\) is the Legendre transform.
\end{prop}

\begin{rmk}
	If the functions \(F\) and \(G\) are smooth, one
	may replace the ``secant lines''
	condition with the derivatives being asymptotically
	arbitrarily close.
	However, the case of interest is piecewise-linear
	continuous functions, and we need control
	on all possible tangent lines, so we use the
	more technical condition in the proposition.
\end{rmk}

\begin{proof}
	Every tangent line can be written as a limit of
	secant lines, so the condition on secant lines
	implies that the slopes of 
	lines tangent to \(f\) and \(g\) 
	are arbitrarily close.
	This guarantees (based on the 
	slope interpretation of the Legendre transform)
	that the \(\leg(F)\) is asymptotically arbitrarily close
	to \(\leg(G)\) as \(t \to 0\).
	We are then done by Lemma \ref{lem:aac:asymeq}.
\end{proof}

\begin{proof}
	[Proof of Theorem \ref{thm:du:uncountable}]
	Let \(\newt\) be the 
	Newton polygon formalism on \(R\).

	We may use the same proof as Du, see
	Section 4 of \cite{DuUncountable}.
	In particular, we may define ideals \(I_{\lambda}\) 
	in an analoguous way, given an analogue of
	Lemma 11 of \cite{DuUncountable};
	that is, we only need the following: 
	for \(0 < \mu < 1\),
	we need 
	elements \(g_{\mu} \in R\), 
	such that
	\(\LN(g_{\mu}) \in \omega(t^{\mu})\) 
	but \(\LN(g_{\mu}) \notin O^{\sup}(t^{\lambda})\) 
	for \(0 < \mu < \lambda < 1\). 
	Once we have these elements, 
	by the proof of Proposition 14
	in \cite{DuUncountable}, the ideals
	\(I_{\lambda}\) then satisfy the conditions
	of Kang-Park's theorem, showing 
	the uncountability of the Krull dimension of \(R\).


	By Proposition \ref{prop:leg:asymp},
	it is enough to show that there exists \(g_{\mu}\) 
	such that \(\newt(g_{\mu})\) is 
	asymptotic to \(\leg^{-1}(t^{\mu})\).
	By a standard Legendre transformation calculation
	\(\leg^{-1}(t^{\mu}) = c_{\mu}x^{-r_{\mu}}\)
	for positive constants \(c_{\mu}\) and \(r_{\mu}\).
	Then, as \(\newt\) discretely approximates
	all convex functions with limit zero, 
	it discretely approximates \(\leg^{-1}(t^{\mu})\).
	By a short calculation, we can see that
	the piecewise-linear function with integral
	values at \(f\) is asymptotic to \(f\) for any
	\(f\).
	Thus, we have such an \(g_{\mu}\) and the theorem
	is proved.
\end{proof}

\section{Uncountability of the Krull Dimension of Various Rings} 

\begin{cor}[Kang-Park, Du]
	Let \(V\) be a nondiscrete valuation ring.
	Then \(V[[x]]\) has uncountable Krull dimension.
\end{cor}

\begin{proof}
	We have the obvious Newton polygon formalism.
	To use Theorem \ref{thm:du:uncountable}
	(Du's argument),
	we need to show that 
	this Newton polygon formalism discretely approximantes
	(see \ref{def:discrete:approx})
	any (convex) function with
	limit zero.
	This follows quickly from nondiscreteness of the
	valution on \(V\), we will spell the argument
	out a bit here as all other results of this type will follow
	by a verbatim argument.

	Take a (convex) function \(f\) in \(\contposreal\) which has limit
	zero as \(t \to \infty\). 
	Let \(g\) be the discrete approximation of this function,
	we wish to construct a Newton polygon which 
	is asymptotically arbitrarily close to \(g\),
	which has secant lines which are also asymptotically
	arbitrarily close to \(g\).
	Since the valuation on \(V\) is nondiscrete, one may
	simply choose coefficients \(a_{i}\) which are
	arbitrarily close to \(f(i) (= g(i))\). 
	By choosing close enough \(a_{i}\), one verifies that
	Newton polygon of the element
	\(r = \sum_{i}^{} a_{i}x^{i} \) 
	has the desired properties.
\end{proof}

Note that \(r\) in the previous proof 
may be an infinite sum, so we are 
using the fact that \(V[[x]]\) is
\(x\)-adically complete.

\begin{cor}[Du]
	\(\mathbb{A}_{\inf}\) has uncountable Krull dimension.
\end{cor}

\begin{proof}
	The Gauss norms are multiplicative
	(again see \cite[Chapter 3]{FFCourbes})
	and the norm is nondiscrete on the image of the 
	multiplicative lift.
	Thus, we have a Newton polygon formalism on
	\(\mathbb{A}_{\inf}\), which discretely
	approximates any function with limit zero
	(by the nondiscreteness of the valuation
	on the image of the multiplicative lift
	and the fact that \(\mathbb{A}_{\inf}\)
	is \(\pi\)-complete).
	And we conclude by Theorem \ref{thm:du:uncountable}.
\end{proof}

We have now recreated the results from \cite{DuUncountable},
so we move on to new results:

\begin{lem} \label{lem:coperf:uncountable}
	Let \(V = W(\residue)\coperfx\) 
	(resp. \(\residue\coperfx\))
	for some perfect field \(\residue\) of characteristic \(p\).
	Then
	the \(p\)-adic completion
	(resp. \(t\)-adic completion) of
	\(V\malcevp\) (resp. \(V\malcevt\))
	has uncountable Krull dimension.
\end{lem}

\begin{proof}
	\(V\) has an obvious family of \(p\)-adic valuations
	(resp. has an obvious valuation).
	Furthermore, in the
	\(p\)-adic case \(V\) embeds into \(W(V / p)\),
	so by Corollary \ref{cor:malcevp:mult}
	the Gauss norm on \(\widehat{V}\malcevp\)
	(resp. \(V\malcevt\)) is multiplicative.
	Note that to use 
	Corollary \ref{cor:malcevp:mult}
	we must use \(p\)-adic completion 
	of \(V\)
	in the \(p\)-adic case. 
	We may replace \(V\) with \(\widehat{V}\) without
	loss of generality as
	\[
	\widehat{V\malcevp} \isom \widehat{\widehat{V}\malcevp}
	.\] 

	Next, by Proposition \ref{prop:constr:npf},
	\(\widehat{V}\malcevp\) (resp. \(V\malcevt\))
	has a Newton polygon formalism.
	This extends to the \(p\)-adic completion
	(resp. \(t\)-adic completion)
	by Lemma \ref{lem:npf:completion}.
	As we have arbitrary \(p\)-th roots of \(x\),
	this valuation is clearly nondiscrete
	on elements of the image of the
	multiplicative lift
	(resp. the valuation is clearly nondiscrete
	on \(V\)).
	Thus, this formalism discretely approximates 
	any function with limit zero as \(s \to \infty\) 
	and we conclude by Theorem \ref{thm:du:uncountable}.
\end{proof}

We will use the previous lemma to show that 
the perfectoid Tate algebra of dimension one
has uncountable Krull dimension, by embedding
into the ring described in the lemma.
For our other applications,
we will do a similar trick with different rings:

\begin{lem} \label{lem:puiseux:uncountable}
	Let \(V = W(\residue)\puiseuxx\).
	Then the \(p\)-adic completion of
	\(V\malcevp\)
	has uncountable Krull dimension.
\end{lem}

\begin{proof}
	We use the same proof as Lemma \ref{lem:coperf:uncountable}.
	Again, \(V\) has an obvious \(p\)-adic family
	of valuations,
	and embeds into \(W(V / p)\) by the
	obvious map.
	Again, the valuation is nondiscrete,
	so the formalism discretely approximates
	any function with limit zero.
\end{proof}

\begin{lem}
	\label{lem:coperfxpow:uncountable}
	Let \(V = W(\residue)\coperfxpow\).
	Then the \(p\)-adic completion of \(V\malcevp\)
	has uncountable Krull dimension.
\end{lem}

\begin{proof}
	Analogous proof to \ref{lem:puiseux:uncountable}
	and \ref{lem:coperf:uncountable}.
\end{proof}

\subsection{The perfectoid Tate algebra over \(\mathcal{O}_{K}\)}

Let \(K\) be a perfectoid field of characteristic \(p\).
Suppose to begin that \(K\) 
is spherically complete.
Then
\(K \isom \residue\malcevt\).
Pick \(t\) as our pseudo-uniformizer.
Then the completion of the ring 
\(\residue\malcevt\coperfx\)
is isomorphic to \(\tateone{K}\).
We consider the obvious map
\[
i \colon \residue\malcevt\coperfx \inj \residue\coperfx\malcevt
.\] 
In the latter ring, the ``roles of \(x\) and \(t\) are swapped''.
We can thus use Lemma \ref{lem:npf:injection} to conclude
that the ring on the left has a Newton polygon formalism,
as we know that the one on the right does from
the proof of \ref{lem:coperf:uncountable}.
We need only show that this formalism discretely approximates
any function with limit zero.
We see that as the Mal'cev-Neumann series variable \(t\) has
a preimage in \(\residue\malcevt\coperfx\), and
that the valuation remains nondiscrete on the preimage
\(\residue\coperfx\),
this indeed holds and we may conclude
that \(\tateone{K}\) has uncountable Krull dimension.

Now, let \(K\) be an arbitrary 
perfectoid field of characteristic \(p\),
use the following chain of inclusions:
\[
K\inj K^{alg} \inj K^{sp} \isom \residue\malcevtrat
.\] 

By \cite{KedPoon}, Theorem 3.3, there always exists
an embedding from \(K^{alg} \to K^{sp}\) which takes
a given pseudo-uniformizer \(\pi\) to the variable
\(t\) (simply post-compose with an automorphism of
\(\residue\malcevtrat\)),
where \(\kappa\) is the residue field of \(K^{alg}\).
Moreover, if we choose a pseudo-uniformizer \(\pi \in K\),
it remains a pseudo-uniformizer in a fixed algebraic
closure.

\begin{thm}
	\label{thm:integraltate:uncountable}
	Let \(K\) be a perfectoid field of characteristic
	\(p\).
	Then \(\tateone{K}\) has uncountable Krull dimension.
\end{thm}

\begin{proof}
	Let \(\residue\) be the residue field of \(K^{alg}\),
	and let \(\residue_{K} \in \residue\) be the residue
	field of \(K\).
	Then, by the spherically complete case above, 
	\(\mathcal{O}_{K^{sp}}\coperfx\) 
	has a Newton polgon formalism,
	and we see that along the above mentioned embedding
	\(\mathcal{O}_{K}\coperfx \) has a
	Newton polygon formalism for which 
	\(\pi\) has a preimage.
	Since the Gauss norm remains nondiscrete
	on elements which \(p\) does not divide,
	and \(\tateone{K}\)
	is the completion of
	\(\mathcal{O}_{K}\coperfx \),
	we conclude by Theorem \ref{thm:du:uncountable}.

\end{proof}

If \(K\) has characteristic zero, then
we can run a similar argument:
\[
	K \inj K^{alg} \inj K^{sp} \isom W(\residue)\malcevp
,\] 
similarly to the characteristic \(p\) case.
Notice that \(p\) is always a pseudo-uniformizer.

\begin{thm}
	Let \(K\) be a perfectoid field of characteristic zero.
	Then \(\tateone{K}\) has uncountable Krull dimension.
\end{thm}

\begin{proof}
	Same argument as in characteristic \(p\),
	just replace \(t\) with \(p\),
	\(\residue\) with \(W(\residue)\), and
	\(\residue_{K}\) with \(W(\residue_{K})\).
	Since \(p\) is always a pseudo-uniformizer,
	we do not need to compose with an automorphism
	of \(K^{sp}\).
\end{proof}

Finally, we extend the proof to higher dimensional
Tate algebras:

\begin{cor} \label{cor:extend:triv}
	Let \(K\) be a perfectoid field.
	Then the perfectoid Tate algebra in \(n\) variables
	has uncountable Krull dimension.
\end{cor}

\begin{proof}
	We use the surjective
	ring map 
	\[
	\mathcal{O}_{K}\langle x_{1}^{1 / p^{\infty}}, \ldots, x_{n}^{1 / p^{\infty}}\rangle 
	\to \tateone{K}
	,\] 
	which, say, evaluates \(x_{2}, \ldots, x_{n}\) at zero
	and pull back the ideals.
\end{proof}

\subsection{The perfectoid Tate algebra over \(K\)}

Let \(R = \tateone{K}\) be the perfectoid
Tate algebra in one variable over \(\mathcal{O}_{K} \).
We will show that \(p\) is not contained in any of the
prime ideals coming from the Kang-Park theorem,
and deduce that \(K\langle x^{1 / p^{\infty}} \rangle\) 
has uncountable dimension.

Following Du, let \(\mathfrak{m}\) be the ideal
of elements for which every coefficient has 
positive valuation. More precisely,

\[
\mathfrak{m} = \{\sum_{0 \leq i}^{} a_{i}\pi^{i} | 0 < v(a_{i}) \text{ for all } i \}
,\] 

where \(\pi\) is the chosen pseudo-uniformizer from the previous section
and the valuation \(v\) is the \(x\)-adic one.
We can see by construction that \(\mathfrak{p} \ins \mathfrak{m}\) 
and that all 
\(g_{\mu}\) are in \(\mathfrak{m}\), so 
\(I_{\lambda} \ins \mathfrak{m}\) for all \(0 < \lambda < 1\)

Now, recall the following from the proof of Theorem 3
of \cite{KangPark}: the prime ideals
\(p_{\lambda}\) coming from \(I_{\lambda}\) are defined
by taking a certain valuation ring \(W\) which contains
\(R\), extending the ideals \(I_{\lambda}\) to \(W\),
taking radicals, and restricting.
In other words,

\[
	\mathfrak{p}_{\lambda} = \sqrt{I_{\lambda}W} \cap R
.\] 

Note that as \(\mathfrak{m}\) is prime,
we have 
\(\sqrt{\mathfrak{m}W} \cap R \ins \mathfrak{m}\),
and in particular \(p_{\lambda} \ins \mathfrak{m}\) 
for all \(\lambda\).
We have shown

    
\begin{lem}
	 Let \(\pi\) be the chosen pseudo-uniformizer
	 from the previous section 
	 (i.e. whichever arbitrarily chosen one
	  in characteristic \(p\), always \(p\) 
	  in characteristic \(0\)).
     Any polynomial in \(\pi\) 
     that has a coefficient of valuation 0 is not 
     contained in \(\mathfrak{m}\), and thus
     is not contained in any other of the \(\mathfrak{p}_{\lambda}\)
     for \(0 < \lambda \leq 1\).
\end{lem}
    
    

The following corollary is an immediate specialization:

\begin{cor} \label{cor:field:uncountable}
    The element \(\pi\) is not in
    any of the \(\mathfrak{p}_{\lambda}\).
	Thus, the perfectoid ring \(R[1 / p]\)
	has uncountable Krull dimension.
	So the perfectoid Tate algebra over \(K\) in
	\(n\) variables
	has uncountable Krull dimension, regardless of characteristic.
\end{cor}

\subsection{The absolute integral closure in mixed characteristic}

Let \((A, \mathfrak{m}, k)\) be a noetherian local domain of mixed characteristic
which is \(\mathfrak{m}\)-adically complete. 
We are concerned with the absolute integral closure \(A^{+}\) of \(A\).
In \cite{HeitmannRplus}, Heitmann asks what the Krull dimension of 
\(A^{+}\) is. 
We form the following conjectural answer.

\begin{conj}
	\label{conj:rplus:uncountable}
	Let \(\dim A = 2\).
	Then the Krull dimension of 
	the \(p\)-adic completion of \(A^{+}\) is uncountable.
\end{conj}

Knowing this conjecture would be enough to answer Heitmann's question:

\begin{cor}
	Assume that Conjecture \ref{conj:rplus:uncountable}
	holds. 
	Let \(2 \leq \dim A\).
	Then the Krull dimension of 
	the \(p\)-adic completion of \(A^{+}\) is
	uncountable.
\end{cor}

\begin{proof}
	WLOG \(A \isom W(\residue)[[x_{1}, \ldots, x_{n}]]\) 
	where \(n = \dim A - 1\).
	Pick a prime ideal \(\mathfrak{p}\) in \(A^{+}\) lying over
	\((x_{2}, \ldots, x_{n})\) (which exists
	as the extension is integral).
	Then, using e.g. \cite{HunekeRplusSurvey} (see Discussion 2.5),
	we have that 
	\(A^{+} / \mathfrak{p} \isom A_{2}^{+}\)
	where \(A_{2} := W(\residue)[[x_{1}]] \).
	So we have a surjective ring map 
	\(A^{+} \to A_{2}^{+}\).
	Now, as \(p\)-adic completion preserves surjections,
	we get a surjection \(\widehat{A^{+}} \to \widehat{A_{2}^{+}}\).
	By the above theorem, 
	\(\widehat{A_{2}^{+}}\) has uncountable Krull dimension,
	pull back the ideals.
\end{proof}

The motivation for this conjecture is due to the following classical 
theorem:

\begin{thm}
	[Newton-Puiseux, see \cite{AbhyankarSciEng} Lecture 12]
	Let \(K\) be an algebraically closed field of characteristic zero.
	Then \(\overline{K((X))} \isom K\puiseuxxrat\).
	Moreover, if
	\(R = \mathcal{O}_{K}[[x]]\),
	then \(R^{+} \ins K\puiseuxx\).
\end{thm}

This gives a way to explicitly write down elements of \(A^{+}\) in
the dimension 2 case
(note that by the Cohen Structure Theorem it is enough to show 
Conjecture \ref{conj:rplus:uncountable} 
for \(A = R\)). 
One might hope to leverage this to answer Conjecture \ref{conj:rplus:uncountable}.
However, with our machinery so far we are only able to show the following.

\begin{prop}
	Let \(K\) be an algebraically closed field with residue characteristic \(p\).
	Then the \(p\)-adic completion of
	\(S = \mathcal{O}_{K}\puiseuxx \) has uncountable Krull dimension.
	Moreover \(\widehat{S}[1 / p]\) has uncountable Krull dimension.
\end{prop}

\begin{proof}


	We first show that 
	\(\mathcal{O}_{K}\puiseuxx\)
	has uncountable Krull dimension.
	As before, we can embed into a spherical completion,
	letting \(\residue = \overline{\mathbb{F}_{p}}\)
	and \(\Gamma = \mathbb{Q}\), we get
	\[
	\mathcal{O}_{\mathbb{C}_{p}}\puiseuxx
	\inj
	\residue\malcevp\puiseuxx
	.\] 
	We may now use the same argument as for the perfectoid
	Tate algebra, with Lemma \ref{lem:coperf:uncountable}
	replaced by Lemma \ref{lem:puiseux:uncountable}.
	We need only to show that 
	the Mal'cev-Neumann series variable \(p\) pulls back
	and that the norm remains nondiscrete on coefficients.
	But these both follow as before,
	and these two things with \(p\)-adic completeness
	together imply that the \(p\)-adic completion
	of the ring on the left has a Newton polygon 
	formalism which discretely approximates
	any function with limit zero.
	By Theorem \ref{thm:du:uncountable}, we conclude.

	Finally, we may use the same argument as 
	\ref{cor:field:uncountable} to invert \(p\).
\end{proof}

\subsection{The case of \(A_{\infty}\)}

We now consider the ring \(A_{\infty}\) as defined
in the paper \cite{PerfectoidSignature}.
Recall the construction, 
let \(A = W(k)[[x_{2}, \ldots, x_{d}]]\) for
\(k\) a perfect field of characteristic \(p\).
Then
\[
	A_{\infty,d} = A[p^{1 / p^{\infty}}, x_{2}^{1 / p^{\infty}}, \ldots, x_{d}^{1 / p^{\infty}}]
.\] 

\begin{thm}
	\(A_{\infty,2}\) has uncountable Krull dimension.
\end{thm}

\begin{proof}
	Use the same proof as 
	Theorem \ref{thm:integraltate:uncountable},
	but with Lemma \ref{lem:coperfxpow:uncountable} 
	instead of Lemma \ref{lem:coperf:uncountable}.
	%
	%
\end{proof}

\begin{cor}
	Let \(2 \leq d\). Then
	\(A_{\infty, d}\) has uncountable Krull dimenison.
\end{cor}

\begin{proof}
	As with the perfectoid tate algebra, 
	we have a surjection \(A_{\infty,d} \to A_{\infty,2}\).
\end{proof}

\bibliographystyle{alpha}
\bibliography{main}

\newcommand{\etalchar}[1]{$^{#1}$}
\begin{thebibliography}{BCKW19}

\bibitem[Abh90]{AbhyankarSciEng}
Shreeram~Shankar Abhyankar.
\newblock {\em Algebraic geometry for scientists and engineers}.
\newblock Mathematical surveys and monographs, number 35. American Mathematical
  Society, Providence, R.I, 1990 - 1990.

\bibitem[And18]{AndreDirectSummand}
Yves Andr{\'e}.
\newblock La conjecture du facteur direct.
\newblock {\em Publications math{\'e}matiques de l'IH{\'E}S}, 127(1):71--93,
  2018.

\bibitem[BCKW19]{AWSNotes}
Bhargav Bhatt, Ana Caraiani, Kiran~S. Kedlaya, and Jared Weinstein.
\newblock {\em Perfectoid spaces}, volume 242 of {\em Mathematical Surveys and
  Monographs}.
\newblock American Mathematical Society, Providence, RI, 2019.
\newblock Lectures from the 2017 Arizona Winter School, held in Tucson, AZ,
  March 11--17, Edited and with a preface by Bryden Cais, With an introduction
  by Peter Scholze.

\bibitem[BIM19]{BhattIyengarMa}
Bhargav Bhatt, Srikanth~B. Iyengar, and Linquan Ma.
\newblock Regular rings and perfect(oid) algebras.
\newblock {\em Communications in Algebra}, 47(6):2367--2383, 2019.

\bibitem[CLM{\etalchar{+}}23]{PerfectoidSignature}
Hanlin Cai, Seungsu Lee, Linquan Ma, Karl Schwede, and Kevin Tucker.
\newblock Perfectoid signature, perfectoid hilbert-kunz multiplicity, and an
  application to local fundamental groups, 2023.

\bibitem[CS17]{CarianiScholze1}
Ana Caraiani and Peter Scholze.
\newblock On the generic part of the cohomology of compact unitary {S}himura
  varieties.
\newblock {\em Ann. of Math. (2)}, 186(3):649--766, 2017.

\bibitem[Du20]{DuUncountable}
Heng Du.
\newblock $\mathbf{A}_{\text {inf}}$ has uncountable krull dimension, 2020.

\bibitem[FF18]{FFCourbes}
Laurent Fargues and Jean-Marc Fontaine.
\newblock Courbes et fibr\'{e}s vectoriels en th\'{e}orie de {H}odge
  {$p$}-adique.
\newblock {\em Ast\'{e}risque}, (406):xiii+382, 2018.
\newblock With a preface by Pierre Colmez.

\bibitem[FS21]{FarguesScholze}
Laurent Fargues and Peter Scholze.
\newblock Geometrization of the local langlands correspondence, 2021.

\bibitem[Hei22]{HeitmannRplus}
Raymond~C. Heitmann.
\newblock {$\widehat {R^+}$} is surprisingly an integral domain.
\newblock {\em J. Pure Appl. Algebra}, 226(1):Paper No. 106809, 7, 2022.

\bibitem[Hub96]{HuberAdicSpaces}
Roland Huber.
\newblock {\em \'{E}tale cohomology of rigid analytic varieties and adic
  spaces}.
\newblock Aspects of Mathematics, E30. Friedr. Vieweg \& Sohn, Braunschweig,
  1996.

\bibitem[Hun11]{HunekeRplusSurvey}
Craig Huneke.
\newblock Absolute integral closure.
\newblock In {\em Commutative algebra and its connections to geometry}, volume
  555 of {\em Contemp. Math.}, pages 119--135. Amer. Math. Soc., Providence,
  RI, 2011.

\bibitem[Kap42]{KaplanskyMaximal1}
Irving Kaplansky.
\newblock {Maximal fields with valuations}.
\newblock {\em Duke Mathematical Journal}, 9(2):303 -- 321, 1942.

\bibitem[Kap45]{KaplanskyMaximal2}
Irving Kaplansky.
\newblock Maximal fields with valuations. {II}.
\newblock {\em Duke Math. J.}, 12:243--248, 1945.

\bibitem[KL15]{KedlayaLiu1}
Kiran~S. Kedlaya and Ruochuan Liu.
\newblock Relative {$p$}-adic {H}odge theory: foundations.
\newblock {\em Ast\'{e}risque}, (371):239, 2015.

\bibitem[KP05]{KedPoon}
Kiran~S. Kedlaya and Bjorn Poonen.
\newblock Orbits of automorphism groups of fields.
\newblock {\em Journal of Algebra}, 293(1):167--184, 2005.

\bibitem[KP13]{KangPark}
B.G. Kang and M.H. Park.
\newblock Krull-dimension of the power series ring over a nondiscrete valuation
  domain is uncountable.
\newblock {\em Journal of Algebra}, 378:12--21, 2013.

\bibitem[LL21]{LangLudwig}
Jaclyn Lang and Judith Ludwig.
\newblock {$\Bbb{A}_{\inf}$} is infinite dimensional.
\newblock {\em J. Inst. Math. Jussieu}, 20(6):1983--1989, 2021.

\bibitem[MS21]{MaSchwedeBCM}
Linquan Ma and Karl Schwede.
\newblock {Singularities in mixed characteristic via perfectoid big
  Cohen–Macaulay algebras}.
\newblock {\em Duke Mathematical Journal}, 170(13):2815 -- 2890, 2021.

\bibitem[Poo93]{PoonThesis}
Bjorn Poonen.
\newblock Maximally complete fields.
\newblock {\em Enseign. Math. (2)}, 39(1-2):87--106, 1993.

\bibitem[Sch12]{ScholzePerfSpaces}
Peter Scholze.
\newblock Perfectoid {Spaces}.
\newblock {\em Publications Math\'ematiques de l'IH\'ES}, 116:245--313, 2012.

\bibitem[Sch13]{ScholzeHodgeRigid}
Peter Scholze.
\newblock {$p$}-adic {H}odge theory for rigid-analytic varieties.
\newblock {\em Forum Math. Pi}, 1:e1, 77, 2013.

\bibitem[Sch14]{ScholzeSurvey}
Peter Scholze.
\newblock Perfectoid spaces and their applications.
\newblock In {\em Proceedings of the {I}nternational {C}ongress of
  {M}athematicians---{S}eoul 2014. {V}ol. {II}}, pages 461--486. Kyung Moon Sa,
  Seoul, 2014.

\bibitem[Sch15]{ScholzeTorsion}
Peter Scholze.
\newblock On torsion in the cohomology of locally symmetric varieties.
\newblock {\em Ann. of Math. (2)}, 182(3):945--1066, 2015.

\end{thebibliography}

\end{document}